%% file: paper1.tex
\documentclass[11pt]{article}
\usepackage{amsthm,amssymb,epsfig,graphicx,amscd,amsmath,color,bm}
\usepackage{graphicx, color,enumerate}
\usepackage{hyperref}
\usepackage[all]{xy}

\let \newline = \par

\theoremstyle{definition}
\newtheorem{theorem}{Theorem}

\newtheorem{corollary}[theorem]{Corollary}

\newtheorem{example}[theorem]{Example}

\newtheorem{open problem}[theorem]{Open Problem}

\newtheorem*{teo}{Theorem 34}

\newtheorem{definition}[theorem]{Definition}

\newtheorem{remark}[theorem]{Remark}
\newtheorem{lemma}[theorem]{Lemma}

\newcommand{\R}{\mathbb{R}}

\newcommand{\Z}{\mathbb{Z}}

\newtheorem{notation}[theorem]{Notation}

\begin{document}

\title{RELATIVE OUTER AUTOMORPHISMS OF FREE GROUPS}

\author{Erika Meucci}
\maketitle

\begin{abstract}
Let $A_1, \ldots, A_k$ be a system of free factors of $F_n$. The
group of relative automorphisms $\mathrm{Aut}(F_n;A_1, \ldots, A_k)$
is the group given by the automorphisms of $F_n$ that restricted
to each $A_i$ are conjugations by
ele\-ments in $F_n$. The group of relative outer automorphisms is defined as
$\mathrm{Out}(F_n;A_1, \ldots, A_k) = \mathrm{Aut}(F_n;A_1, \ldots, A_k) / \mathrm{Inn}(F_n)$,
where $\mathrm{Inn}(F_n)$ is the normal subgroup of $\mathrm{Aut}(F_n)$ given by all
the inner automorphisms.
We define a contractible space on which $\mathrm{Out}(F_n;A_1, \ldots, A_k)$ acts with finite
stabilizers and we compute the virtual cohomological dimension of this group.
\end{abstract}

\section{Introduction}

Let $F_n$ denote the free group on $n$ generators. We consider the
group of automorphisms of $F_n$, denoted by $\mathrm{Aut}(F_n)$, and the
group of outer automorphisms
$$
\mathrm{Out}(F_n) = \mathrm{Aut}(F_n) / \mathrm{Inn}(F_n),
$$
where $\mathrm{Inn}(F_n)$ is the normal subgroup of $\mathrm{Aut}(F_n)$ given by all
the inner automorphisms.
\newline Culler and Vogtmann introduced a space $\mathrm{CV}_n$ on which the
group $\mathrm{Out}(F_n)$ acts with finite stabilizers and proved that $\mathrm{CV}_n$
is contractible. That space $\mathrm{CV}_n$ is called \emph{outer space}.
\newline Let $A_1, \ldots, A_k$ be a system of free factors of $F_n$, i.e., there exists
$B < F_n$ such that $F_n = A_1 * \cdots * A_k * B$. We define the
group of relative (to $A_1, \ldots, A_k$) automorphisms
$\mathrm{Aut}(F_n;A_1, \ldots, A_k)$ given by the elements $f \in \mathrm{Aut}(F_n)$
such that $f$ restricted to each $A_i$ is a conjugation by an
element in $F_n$.
\newline Obviously, $\mathrm{Aut}(F_n) > \mathrm{Aut}(F_n;A_1, \ldots, A_k) \triangleright \mathrm{Inn}(F_n)$.
\newline We define also the group of relative (to
$A_1, \ldots, A_k$) outer automorphisms:
$$
\mathrm{Out}(F_n;A_1, \ldots, A_k) = \mathrm{Aut}(F_n;A_1, \ldots, A_k) / \mathrm{Inn}(F_n) <
\mathrm{Out}(F_n).
$$
We are interested in finding the relative outer space $\mathrm{CV}_n(A_1,
\ldots, A_k)$ on which $\mathrm{Out}(F_n;A_1, \ldots, A_k)$ acts with finite
stabilizers and proving that $\mathrm{CV}_n(A_1,$ \ldots, $A_k)$ is
contractible. Moreover, we will compute the VCD of $\mathrm{Out}(F_n;A_1,
\ldots$ $, A_k)$. Let $s(i)$ be the minimum number of generators for $A_i$.
\begin{teo}
We have
$$
\mathrm{vcd}(\mathrm{Out}(F_n;A_1, \ldots, A_k)) = 2n - 2s(1)- \cdots - 2 s(k) +2k
-2-m,
$$
where $s(i_1) = \cdots = s(i_m) = 1$ and $s(j)>1$ for  $j \neq i_1,
\ldots, i_m$.
\end{teo}
Our computation of the virtual cohomological dimension of the
relative outer space agrees with the computation in \cite{BCV} when
$m=k$. For $k=n$ and $s(1)= \cdots =s(k)=1$, $\mathrm{Out}(F_n;A_1, \ldots, A_k)$ is
called the \emph{pure symmetric automorphism group}. In \cite{C} Collins
showed that the virtual cohomological
dimension of the pure symmetric automorphism group is $n-2$.

\subsection*{Acknowledgments}
I would like to thank my advisor
Mladen Bestvina for his support and his suggestions.

\section{\boldmath $\mathrm{Out}(F_n;A_1, \ldots, A_k)$ and $\mathrm{CV}_n(A_1, \ldots, A_k)$}

The goal of this section is to define carefully $\mathrm{Out}(F_n;A_1,
\ldots, A_k)$ and $\mathrm{CV}_n(A_1,$ $ \ldots, A_k)$.
\newline Consider $A_i=<y_1^i, \ldots,
y_{s(i)}^{i}>$ and $F_n=<y_1^1, \ldots, y_{s(k)}^{k}, x_1, \dots,
x_{n-\sum_{i=1}^{k} s(i)}>$. By a graph, we mean a connected $1$-dimensional CW complex.
\newline Let the \emph{relative rose} $R_n(A_1,
\ldots, A_k)$ be a graph obtained by a wedge of $n-\sum_{i=1}^{k}
s(i)$ circles attaching $\sum_{i=1}^{k} s(i)$ circles $C_1^1,
\ldots, C_{s(k)}^{k}$ on $k$ stems as in Figure~\ref{fig:relrose}.
The edges are denoted by $C_1^1, \ldots, C_{s(k)}^{k}, f_1, \ldots,
f_k, e_1, \ldots, e_{n-\sum_{i=1}^{k} s(i)}$. Moreover,
$$
\pi_1(R_n(A_1, \ldots, A_k), v) \cong F_n = <y_1^1, \ldots,
y_{s(k)}^{k}, x_1, \dots, x_{n-\sum_{i=1}^{k} s(i)}>,
$$
where $v$ is the central vertex in $R_n(A_1, \ldots, A_k)$ (see
Figure~\ref{fig:relrose}), by declaring $y_i^j$ to be the homotopy
class of $C_i^j$ and $x_i$ to be the homotopy class of the loop
$e_i$.

\begin{figure}[htbp]
\begin{center}
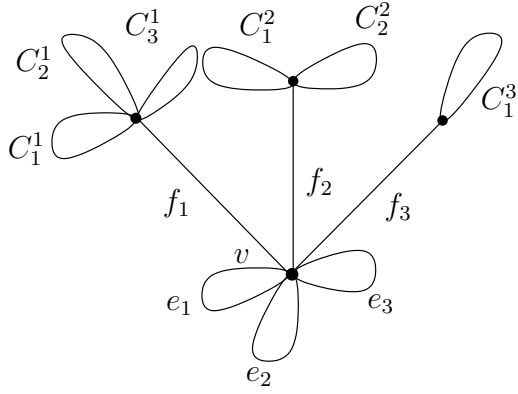
\caption{The relative rose $R_9(A_1, A_2, A_3)$.}
\label{fig:relrose}
\end{center}
\end{figure}

Let $(R_n(A_1,\ldots, A_k), \underline{k})$ be the graph
$R_n(A_1,\ldots, A_k)$ equipped with inclusions $k_j: \bigvee_{i=1}^{s(j)} S^1
\rightarrow R_n(A_1,\ldots, A_k)$ that identifies
$\bigvee_{i=1}^{s(j)} S^1$ with $\bigvee_{i=1}^{s(j)} C_i^j$, for
all $j=1, \ldots, k$.
\begin{definition}
Let $\Gamma$ be a graph of rank $n$ with vertices of valence at
least $3$, equipped with embeddings $l_j: \bigvee_{i=1}^{s(j)} S^1
\rightarrow \Gamma$ for $j=1, \ldots, k$. We call $\mathbb{B}_j =
l_j(\bigvee_{i=1}^{s(j)} S^1)$ \emph{wedge cycle}. The \emph{dual
graph} of the $\mathbb{B}_j$'s is the graph with one vertex for each
wedge cycle, one vertex $w$ for each intersection between two or
more wedge cycles and edges between $w$ and vertices corresponding
to the wedge cycles meeting in $w$.
\end{definition}
\begin{definition}
An $(A_1,\ldots, A_k,n)$\emph{-graph} $(\Gamma, \underline{l})$ is a
finite graph $\Gamma$ of rank $n$ with vertices of valence at least $3$,
with possible separating edges, equipped with embeddings $l_j:
\bigvee_{i=1}^{s(j)} S^1 \rightarrow \Gamma$ for $j=1, \ldots, k$,
such that any two $\mathbb{B}_j$ intersect in at most a point and
the dual graph of the $\mathbb{B}_j$'s is a forest.
\end{definition}

\begin{notation}
We will denote by $\mathcal{A}$ the set of free factors $A_1,\ldots, A_k$.
\end{notation}

\begin{example}\label{exnograph}
Consider the graph in Figure~\ref{fig:exnograph}. Each loop of the
same color is a wedge cycle. That graph is not
an $(A_1, A_2, A_3, 4)$-graph because the dual graph of the $\mathbb{B}_j$'s
is a circle (see Figure~\ref{fig:dualgraph}).

\begin{figure}[htbp]
\begin{center}
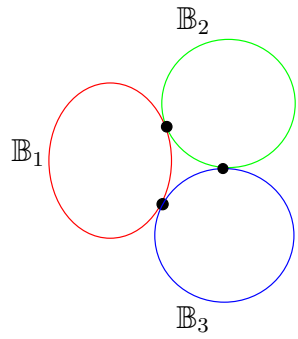

\caption{Example of a graph that is not an $(A_1, A_2, A_3, 4)$-graph.}
\label{fig:exnograph}
\end{center}
\end{figure}

\begin{figure}[htbp]
\begin{center}
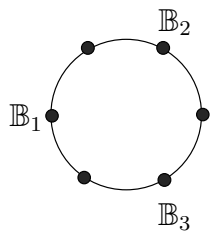

\caption{The dual graph of the graph in Figure~\ref{fig:exnograph}.}
\label{fig:dualgraph}
\end{center}
\end{figure}

\end{example}

\begin{definition}
A \emph{marked} $(\mathcal{A},n)$\emph{-graph} $(\Gamma, \phi)$
is a graph $\Gamma$ of rank $n$ equipped with a homotopy equivalence
$\phi: R_n(\mathcal{A}) \rightarrow \Gamma$ such that $(\Gamma,
\phi \circ \underline{k})$ is an $(\mathcal{A},n)$-graph.
\newline The map $\phi$ is called the \emph{marking}.
\end{definition}
The marking induces an isomorphism $\phi_{*}: F_n \rightarrow
\pi_1(\Gamma, \phi(v))$.
\begin{definition}
A \emph{marked metric $(\mathcal{A},n)$-graph} $(\Gamma,\phi)$
is a marked graph $(\Gamma, \phi)$ such that each edge $e$ in
$\Gamma$ has positive real length $l(e)$.
\end{definition}
\begin{definition} \label{outspace}
The \emph{relative outer space} $\mathrm{CV}_n(A_1,\ldots, A_k)$ (or
$\mathrm{CV}_n(\mathcal{A})$) is the space of equivalence classes of
marked metric $(\mathcal{A},n)$-graphs
where
        \begin{enumerate}
        \item the sum of all lengths of the edges in $\Gamma \setminus
        \{ \phi(C_1^1), \ldots, \phi(C_{s(k)}^{k}) \}$
        is $1$ (relative volume $1$) and
        $$
        \sum_{e \subset \phi(C_i^j)} l(e)=1 \quad \forall \, i, j;
        $$
        \item $(\Gamma_1,\phi_1) \sim (\Gamma_2,\phi_2)$ if there is
        an isometry $h:\Gamma_1 \rightarrow \Gamma_2$ with $h$
        such that $h \circ \phi_1(C_i^j)=\phi_2(C_i^j)$, $\forall \, i,j$,
        and $h \circ \phi_1$ is homotopic to $\phi_2$ rel. $\bigcup_{i,j} C_i^j$.
        \end{enumerate}
\end{definition}
We will usually denote a point in $\mathrm{CV}_n(\mathcal{A})$ by $(\Gamma,\phi)$.
\newline There is a natural right action of $\mathrm{Out}(F_n;\mathcal{A})$ by
homeomorphisms of $\mathrm{CV}_n(\mathcal{A})$: let $X=(\Gamma, \phi) \in
\mathrm{CV}_n(\mathcal{A})$, let $\Psi$ be a relative outer automorphism
and consider a map $\psi:R_n(\mathcal{A}) \rightarrow
R_n(\mathcal{A})$ such that $[\psi_*]=\Psi$ and which is
the identity on $\bigcup_{i,j} C_i^j$. Define
$$
X \cdot \Psi =(\Gamma,\phi) \cdot \Psi = (\Gamma, \phi \circ \psi).
$$
We can define a topology on $\mathrm{CV}_n(\mathcal{A})$ by varying the
lengths of the edges exactly as for outer space (see \cite{HV} for
the definition in the case of outer space).
\newline Because we suppose that the relative volume is $1$ and the
sum of the lengths of the edges in each cycle $\phi(C_i^j)$ is $1$, a
point $(\Gamma, \phi) \in \mathrm{CV}_n(\mathcal{A})$ is in the interior
of a polysimplex that is the product of the simplices $\Delta_i^j$
obtained by varying the lengths of the edges in each cycle
$\phi(C_i^j)$ and the simplex $\sigma$ given by varying the length
of the edges in $\Gamma \setminus \{ \phi(C_1^1), \ldots,
\phi(C_{s(k)}^{k}) \}$. Indeed, if $\Gamma$ has $N_{i,j}$ edges in
$\phi(C_i^j)$ of length $s_{i,j}^{1}, \ldots, s_{i,j}^{N_{i,j}}$ and
$N$ edges in $\Gamma \setminus \{ \phi(C_1^1), \ldots,
\phi(C_{s(k)}^{k}) \}$ of length $t_1, \ldots, t_N$, then
$$
\left.
  \begin{array}{ccc}
    0<s_{i,j}^{k}<1, \, \forall \, i,j,k  & \mbox{ and } & \sum_{k=1}^{N_{i,j}}  s_{i,j}^{k}=1, \\
    0<t_m<1, \, \forall \, m & \mbox{ and } & \sum_{m=1}^{N} t_m=1. \\
  \end{array}
\right.
$$
Let $\Delta_i^j$ be the open simplex determined by varying the
$s_{i,j}$'s and $\sigma$ be the open simplex obtained by varying the
$t$'s. Define
$$
P\Delta = \Delta_1^1 \times \cdots \times \Delta_{s(k)}^{k} \times
\sigma.
$$
Changing the length of the edges in $\Gamma$ gives the open
polysimplex $P\Delta$.
\begin{example}\label{expoly}
Consider $\mathrm{Out}(F_5;A_1,A_2)$, where $F_5 = <a,a',b,b',c>$,
$A_1=<a,a'>$ and $A_2=<b,b'>$, and consider the point $(\Gamma,\phi)
\in \mathrm{CV}_5(A_1,A_2)$ in Figure~\ref{fig:expoly1}.

\begin{figure}[htbp]
\begin{center}
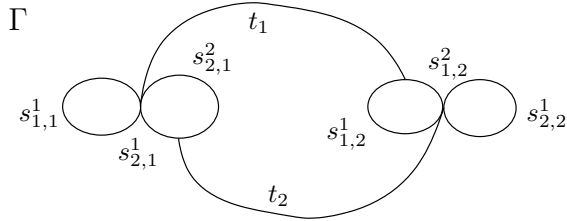

\caption{The point $(\Gamma,\phi) \in \mathrm{CV}_5(A_1,A_2)$ in
Example~\ref{expoly}.} \label{fig:expoly1}
\end{center}
\end{figure}

The open polysimplex given by varying the length of the edges of
$\Gamma$ is the open cube $\Delta_1^1 \times \Delta_2^1 \times
\Delta_1^2 \times \Delta_2^2 \times \sigma \cong \Delta_2^1 \times
\Delta_1^2 \times \sigma$ (see Figure~\ref{fig:excube1}).

\begin{figure}[htbp]
\begin{center}
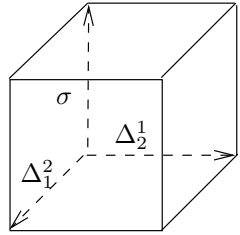

\caption{The open polysimplex $\Delta_2^1 \times \Delta_1^2 \times
\sigma$ in Example~\ref{expoly}.} \label{fig:excube1}
\end{center}
\end{figure}

\end{example}
Note that $\mathrm{Out}(F_n;\mathcal{A})$ acts properly and
discontinuously on $\mathrm{CV}_n(\mathcal{A})$ and that the stabilizer
of any marked $(\mathcal{A},n)$-graph $(\Gamma, \phi)$ is
isomorphic to the subgroup of isometries of $\Gamma$ which fixes the
wedge cycles, hence it is finite.
\newline For an $(\mathcal{A},n)$-graph $(\Gamma, \underline{l})$, let
$\widehat{\Gamma}$ be the graph of rank $n- \sum_{i=1}^{k} s(i)$
obtained from $\Gamma$ by collapsing the wedge cycles $\mathbb{B}_1,
\ldots, \mathbb{B}_k$ to points. Let $e$ be an edge of $\Gamma$
that does not define a loop in $\Gamma$ or in $\widehat{\Gamma}$.
Then the composition with the edge collapse $\mathrm{col}:\Gamma \rightarrow
\Gamma/e$ induces embeddings
$$
l_j/e: \bigvee_{i=1}^{s(j)} S^1 \rightarrow \Gamma/e
$$
such that $(\Gamma/e, \underline{l}/e)$ is again an $(\mathcal{A},n)$-graph.
By an edge collapse in an $(\mathcal{A},n)$-graph,
we will always mean the collapse satisfying the above
hypothe\-sis. For a marked $(\mathcal{A},n)$-graph $(\Gamma, \phi)$, the marking of
the collapsed graph is the composition $\mathrm{col} \circ \phi$.
\newline $F$ is a forest in a graph $\Gamma$ if it is a union of
edges in $\Gamma$ that does not contain any loop in
$\Gamma$ or in $\widehat{\Gamma}$. A forest collapse in an
$(\mathcal{A},n)$-graph $(\Gamma, \underline{l})$ is a sequence
of edge collapses, where the edges that are collapsed are the edges
in the forest. We denote the collapsed graph by $(\Gamma/F,
\underline{l}/F)$.
\newline We define a poset structure on the set of marked $(\mathcal{A},n)$-graphs by
saying that $(\Gamma_1, \phi_1) \leq (\Gamma_2, \phi_2)$ if there is
a forest $F$ in $\Gamma_2$ such that $(\Gamma_2/F, \phi_2/F)$ is
equivalent to $(\Gamma_1, \phi_1)$.
\newline We denote by $S_n(A_1,\ldots, A_k)$ (or $S_n(\mathcal{A})$) the geometric
realization of that poset and we call it the \emph{relative spine}
of the relative outer space.

\begin{example} \label{ex1}
Consider $n=2$, $F_2=<a,b>$ and $A=<a>$.
\newline In that case, $\mathrm{Out}(F_2;A)$ is isomorphic to the infinite dihedral
group $D_{\infty}$. Indeed, if $f \in \mathrm{Aut}(F_2)$ and $f(a)=a$, then
$f(b)=a^n b^{\varepsilon} a^m$, where $n$, $m \in \Z$ and
$\varepsilon \in \{ \pm 1 \}$.
\newline After conjugating by a power of $a$, we have $f(b)=a^N
b$ or $f(b)=a^N b^{-1}$. The map
$$
\left.
    \begin{array}{ccc}
      \mathrm{Out}(F_2; A) & \rightarrow & \Z_2 \\
      f & \mapsto & \varepsilon \\
    \end{array}
  \right.
$$
has kernel $\Z$, so we get
$$
\left.
    \begin{array}{ccccccccc}
      1 & \rightarrow & \Z & \rightarrow & \mathrm{Out}(F_2; A) & \rightarrow & \Z_2 & \rightarrow & 1, \\
    \end{array}
  \right.
$$
i.e. $\mathrm{Out}(F_2; A) \cong \Z \rtimes \Z_2 \cong D_{\infty}$.
\newline The relative spine $S_2(A)$ is homeomorphic to the
simplicial complex in Fi\-gu\-re~\ref{fig:spineex1}.

\end{example}

\begin{figure}[htbp]
\begin{center}
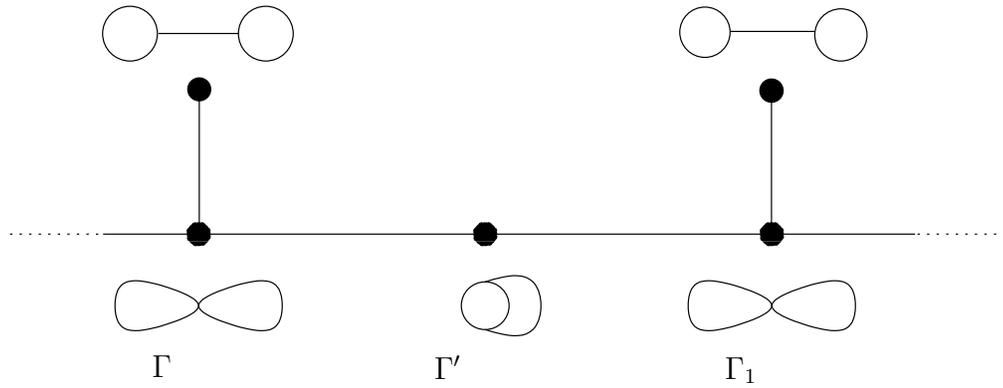

\caption{The spine $S_2(A)$ in Example~\ref{ex1} where $(\Gamma,
id)$ is the marked graph with identity marking, $(\Gamma_1, \phi_1)$
has the marking given by $\phi_1(a)=a$, $\phi_1(b)=ab$. In $\Gamma$
and $\Gamma_1$ both the edges have length 1, while in $\Gamma'$ the
right edge has length $1$ (it corresponds to $\phi'(e_1)$) and the
other edges have length $\frac{1}{2}$.} \label{fig:spineex1}
\end{center}
\end{figure}

\begin{remark}
We can define the reduced relative spine as the subset of
the geometric realization of the poset structure described
previously, containing only the $(\mathcal{A},n)$-graphs
with no separating edges. In Example~\ref{ex1} the reduced relative
spine is homeomorphic to a line.
\end{remark}

Notice that $\mathrm{CV}_n(\mathcal{A})$ is not a polysimplicial complex,
because some of the faces are missing, but the relative spine
$S_n(\mathcal{A})$ is a simplicial complex.
\newline Given a polysimplex $P \Delta$ we define its barycentric subdivision
in the following way. Consider the centroid of each face or polyface (i.e.
product of faces) of the polysimplex. Those will be the vertices of the
barycentric subdivision and we will call them barycentric subdivision
vertices (BSV). Now, for each $n>0$ and $n$-face or $n$-polyface $F$,
connect the centroid with each BSV in the $(n-1)$-faces (or polyfaces)
of $\partial F$ (see Figure~\ref{fig:barycentricsq}).

\begin{figure}[htbp]
\begin{center}
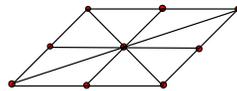

\caption{Barycentric subdivision of the polysimplex given by the product of
two 1-simplices. The dots are the vertices of the barycentric
subdivision.} \label{fig:barycentricsq}
\end{center}
\end{figure}

Note that if $\Delta_j$ is a simplex face in $P \Delta$, then the
restriction of the barycentric subdivision to $P \Delta_{| \Delta_j}$
is the (standard) barycentric subdivision of $\Delta_j$.
\newline There is a natural embedding of
$S_n(\mathcal{A})$ into $\mathrm{CV}_n(\mathcal{A})$ that sends each
vertex to the centroid of the corresponding open polysimplex and
each $d$-simplex to the convex hull of the corresponding centroids
(see Figure~\ref{fig:barycentric} for an example of barycentric
subdivision in $\mathrm{CV}_5(A_1,A_2)$ of Example~\ref{expoly}).

\begin{figure}[htbp]
\begin{center}
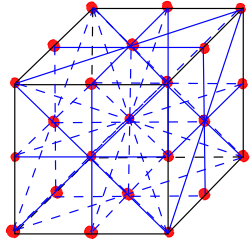

\caption{Barycentric subdivision of the polysimplex given in
Example~\ref{expoly}, where the red dots are the vertices of the simplices
in $S_5(A_1,A_2)$.} \label{fig:barycentric}
\end{center}
\end{figure}

$\mathrm{CV}_n(\mathcal{A})$ deformation retracts onto $S_n(\mathcal{A})$
in the following way. The vertices of $S_n(\mathcal{A})$
correspond to open polysimplices of $\mathrm{CV}_n(\mathcal{A})$ and a
$d$-simplex is a chain of $d+1$ open polysimplices, each of which is
a face of the next. By pushing within each open polysimplex of $\mathrm{CV}_n(\mathcal{A})$
away from the missing faces we have a deformation retraction from
$\mathrm{CV}_n(\mathcal{A})$ to $S_n(\mathcal{A})$. In other words,
$\mathrm{CV}_n(\mathcal{A})$ is the union of open polysimplices in a
polysimplicial complex $X$. Note that $S_n(\mathcal{A})$ is the maximal
full subcomplex of the barycentric subdivision of $X$ that is
disjoint from $X \setminus \mathrm{CV}_n(\mathcal{A})$. Collapsing every
simplex in the barycentric subdivision of $X$ to the face of the
simplex contained in $S_n(\mathcal{A})$ gives a deformation
retraction of $\mathrm{CV}_n(\mathcal{A})$ onto $S_n(\mathcal{A})$.
\newline The action of $\mathrm{Out}(F_n;\mathcal{A})$ extends to a simplicial
action on the relative spine $S_n(\mathcal{A})$.

\section{Contractibility of \boldmath $\mathrm{CV}_n(\mathcal{A})$}

The whole section is dedicated to the proof of the following theorem.
\begin{theorem} \label{contract}
The relative outer space $\mathrm{CV}_n(\mathcal{A})$ is contractible.
\end{theorem}
Because there is a deformation retraction from the relative outer
space $\mathrm{CV}_n(\mathcal{A})$ to its spine $S_n(\mathcal{A})$,
it is enough to prove that the relative spine is contractible.
First we prove that if $k=1$, $S_n(\mathcal{A})$ is contractible.
\newline Recall from \cite{HV} that the spine $S_{n}$ of the outer space
$\mathrm{CV}_n$ is a poset of marked graphs of rank $n$, where the marking is
given by a homotopy equivalence from the rose $R_n= \bigvee_{1}^{n}
S^1$. $S_{n}$ is contractible and admits an action of $\mathrm{Out}(F_{n})$.
Edge collapses induce a poset structure on $S_{n}$ with
minimal elements the \emph{reduced} marked graphs, that is roses.
\newline Let $W$ be the set of conjugacy classes of elements in
$F_n$. Let $w_1, \ldots, w_m$ be elements in $W$.
Following \cite{CV}, we define the function $f_{w_i}$ from the set of roses to $\R$ by
$f((R, \phi)) = n l(w_i)$, where $l$ is the length function on $F_n$
associated to $R$. The minset of $f_{w_i}$ is
$$
\mathrm{Minset}(f_{w_i}) = \bigcup_{f_{w_i} \text{ is min at } (R, \phi)}
\mathrm{st}((R, \phi)),
$$
where $\mathrm{st}(X)$ is the star of the point $X$ in $S_n$.
\newline Now, we consider the function $f$ from the set of roses to $\R^m$,
$f= (f_{w_1}, \ldots, f_{w_m})$ and we consider $\R^m$ equipped with
the lexicographic order.
\newline Define $\Lambda$ as the set of roses $(R,\phi)$ that are minimum
in $f= (f_{w_1}, \ldots, f_{w_m})$ with respect to the lexicographic order. Let
$$
\mathrm{Minset}(f) = \bigcup_{(R,\phi) \in \Lambda} \mathrm{st}((R,\phi)).
$$
\begin{remark}\label{inclusion}
If $(R_i,\phi_i)$ are roses in $S_n$ for $i=1,2$, $f((R_1,\phi_1))
\leq f((R_2,\phi_2))$, $(R_2,\phi_2) \in \Lambda$, then
$(R_1,\phi_1) \in \Lambda$.
\end{remark}
A useful lemma that we will need in the sequel is the Poset Lemma.
\begin{theorem}[Poset Lemma]
Let $X$ be a poset and $f: X \rightarrow X$ be a poset map with the
property that $f(x) \leq x$ for all $x \in X$ (or $f(x) \geq x$ for
all $x \in X$). Then $f(X)$ is a deformation retract of $X$.
\end{theorem}
See \cite{Q} for a proof of the Poset Lemma.
\newline Let $(\Gamma, \phi)$ be a marked graph in $S_{n}$ and let $v$ be
a vertex of $\Gamma$. Formally, the notion of \emph{ideal edges} is
defined as in \cite{CV} in terms of partitions. We can think of an
ideal edge $\gamma$ at the vertex $v$ as a partition of the set $E_v$
of half edges of $\Gamma$ terminating at $v$ such that
the blow-up $\Gamma^{\gamma}$ in $v$ is again in $S_{n}$, where
$\Gamma^{\gamma}$ is the graph obtained by pulling the half edges in
$\gamma$ away from $v$ creating a new vertex $v(\gamma)$, a new edge
$\gamma$ that goes from $v(\gamma)$ to $v$ and each half edge $e
\subset \gamma$ is attached to $v(\gamma)$ instead of $v$. Note that
the graph $\Gamma$ can be reobtained by $\Gamma^{\gamma}$ collapsing
$\gamma$.
\newline An \emph{ideal forest} in a reduced marked graph
is a sequence of ideal edges. One can define a poset structure on
the set of ideal forests of a rose $(R, \phi)$ such that the blowing
up induces an isomorphism between that poset and the star of $(R,
\phi)$ in $S_{n}$ (see \cite{CV}). As for $S_n$, we can define ideal edges and
ideal forests for $S_n(\mathcal{A})$.
\newline Let $A$ be a subset of $E_v$. We will denote by $\overline{A}$ the
complement of $A$.
\begin{definition}
Two subsets $A$ and $B$ of $E_v$ are \emph{compatible} if one of the sets
$A \cap B$, $\overline{A} \cap B$, $A \cap \overline{B}$, $\overline{A}
\cap  \overline{B}$ is empty.
\end{definition}
The upper link of a marked $(\mathcal{A},n)$-graph $(\Gamma,
\phi)$ in $S_n(\mathcal{A})$ is a set of marked $(\mathcal{A},n)$-graphs
$(\Gamma', \phi')$ that collapse to $(\Gamma,
\phi)$. Such marked $(\mathcal{A},n)$-graphs are said to be
obtained by blowing up vertices of $\Gamma$ into trees. Notice that
a set of ideal edges is compatible if it corresponds to a
tree.
\newline Let $B(v)$ be the complex whose vertices are ideal edges at $v$ and
whose $i$-simplices are sets of $i+1$ compatible ideal edges.
\begin{definition}
We say that an ideal edge $\gamma$ at a vertex $v \in \Gamma$ is
\emph{legal} if $\Gamma^{\gamma} \in S_n(\mathcal{A})$. We
denote the subcomplex of $B(v)$ spanned by legal ideal edges by
$L(v)$.
\end{definition}
\begin{remark}
An ideal edge is legal if and only if it separates at
most one pair of half edges contained in a wedge cycle
$\mathbb{B}_i$. Indeed, if it separates two pairs of half edges one in
$\mathbb{B}_i$ and the other one in $\mathbb{B}_j$, then it blow ups
to an edge in $\mathbb{B}_i \cap \mathbb{B}_j$ and that contradicts
the definition of marked $(\mathcal{A},n)$-graph.
\end{remark}
We have the following remarkable result.
\begin{theorem}\label{contra}
$\mathrm{Minset}(f)$ is contractible.
\end{theorem}
The proof of that theorem follows from the following theorem in
\cite{CV}.
\begin{theorem}[Culler-Vogtmann]
Let $W'= \{ w_1, \ldots, w_m \}$, where $w_i \in W$ for all $i$.
Then $\mathrm{Minset}(f)$ is a contractible subcomplex of $S_n$, the action
is proper and the quotient $\mathrm{Minset}(f) / Stab(W')$ is finite.
\end{theorem}
An alternative proof is given in Section $4$
of \cite{JW} (in the case of $G=\{ 1 \}$, $\| \Lambda \| =
\mathrm{Minset}(f)$ and $E_{n}^{G} = S_n$).
\begin{lemma}[\cite{FM}]\label{greengraphthm}
If $(\Gamma_1,\phi_1)$ and $(\Gamma_2,\phi_2)$ are two marked graphs
of rank $n$ and $f:\Gamma_1 \rightarrow \Gamma_2$ is a map linear on
edges such that the following diagram
$$
\xymatrix{
 \Gamma_1 \ar[r]^{f}  & \Gamma_2  \\
 R_{n}  \ar[u]^{\phi_1} \ar[ur]_{\phi_2} &  }
$$
commutes up to homotopy, then there is a subgraph of $\Gamma_1$ where the length of
the edges are multiplied by the Lipschitz constant $\mbox{Lip}(f)$ and the
length of all the edges not in the subgraph are multiplied by a
number strictly less than $\mbox{Lip}(f)$.
\end{lemma}
\begin{definition}
Let $(\Gamma_1,\phi_1)$ and $(\Gamma_2,\phi_2)$ be two marked graphs
of rank $n$. Given a map $f \sim \phi_2 \circ \phi_1^{-1}$ linear on edges, we
denote by $\Gamma_f$ the subgraph of $\Gamma_1$ whose edges
are maximally stretched by $\mbox{Lip}(f)$.
\end{definition}
\begin{definition}
Let $(\Gamma_1,\phi_1)$ and $(\Gamma_2,\phi_2)$ be two marked graphs
of rank $n$. A map $f \sim \phi_2 \circ \phi_1^{-1}$ linear on edges is \emph{not
optimal} if there is some vertex of $\Gamma_f$ such that all
the edges of $\Gamma_f$ terminating at that vertex have
$f$-image with a common terminal partial edge. Otherwise, $f$ is
called \emph{optimal}.
\end{definition}
A \emph{turn} in $(\Gamma,\phi)$ is an unordered pair of oriented
edges of $\Gamma$ originating at a common vertex. A turn is
\emph{nondegenerate} if it is defined by distinct oriented edges.
Otherwise, the turn is called \emph{degenerate}.
\newline A map $f: \Gamma \rightarrow
\Gamma$ induces a map $D f$ from the set of oriented
edges of $\Gamma$ to itself by sending an oriented
edge to the first oriented edge in its $f$-image
as long as no edges are collapsed.
We can think of $D f$ as a sort of derivative.
\newline $Df$ induces a map $T f$
on the set of turns in $\Gamma$.
\newline A turn is \emph{illegal} with respect to $f$ if
its image under
some iterate of $T f$ is degenerate. Otherwise, the turn
is called \emph{legal}. For properties of legal and illegal
turns see \cite{BH} or \cite{AK}.
\newline Remember that
$$
A_i=<y_1^i, \ldots, y_{s(i)}^{i}> \mbox{ and } F_n=<y_1^1, \ldots,
y_{s(k)}^{k}, x_1, \dots, x_{n-\sum_{i=1}^{k} s(i)}>.
$$
Consider the function $f$ given by $f=(f_{1}, f_{2}, \ldots,
f_{k})$, where $$f_j=(f_{w_1^j}, \ldots, f_{w_{s(j)}^{j}},
f_{w_{1,2}^{j}}, \ldots, f_{w_{i,l}^{j}}, \ldots,
f_{w_{i,\overline{l}}^{j}}, \ldots)$$ and
$$
\left.
  \begin{array}{llll}
    w_i^j & = & y_i^j, & \mbox{for } i=1, \ldots, s(k), \\
    w_{i,l}^{j} & = & y_i^j  y_{l}^{j}, & \mbox{for all } i<l, \,  i,l=1, \ldots, s(k), \\
    w_{i,\overline{l}}^{j} & = & y_i^j  \overline{y}_{l}^{j}, &  \mbox{for all } i<l, \,  i,l=1, \ldots, s(k). \\
  \end{array}
\right.
$$
\begin{lemma}\label{starose}
Consider $F_{s(j)}=<y_1^j, \ldots, y_{s(j)}^{j}>$. Then
$\mathrm{Minset}(f_j)$ consists of the star of a single rose in $S_{s(j)}$,
and hence is contractible.
\end{lemma}
\begin{proof}
Suppose that $(R,\phi)$ is a rose as in Figure~\ref{fig:rphi} with
\begin{equation}\label{ipotesi}
l(y_i^j) = 1, \, \forall \, i= 1, \ldots, s(j); \, l(y_i^j
y_l^j) = l(y_i^j \overline{y}_l^j) = 2, \, \forall \, i<l, \, i,l=
1, \ldots, s(j).
\end{equation}

\begin{figure}[htbp]
\begin{center}
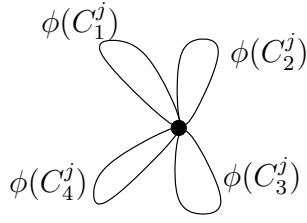

\caption{The rose $(R,\phi)$ with $s(j)=4$.} \label{fig:rphi}
\end{center}
\end{figure}

Let $(R_1,\phi_1)$ be another rose in $\mathrm{Minset}(f_j)$. We will prove
that the optimal homotopy map $f$ such that the following diagram
commutes
$$
\xymatrix{
 R \ar[r]^{f}  & R_1  \\
 R_{s(j)}  \ar[u]^{\phi} \ar[ur]_{\phi_1} &  }
$$
is an isometry up to homotopy. First of all, we need to show that $\Gamma_f = R$.
\newline Notice that by Proposition 3.15 in \cite{FM}, the Lipschitz
constant is determined by a cycle or a figure eight graph. By
(\ref{ipotesi}), $\mbox{Lip}(f)=1$.
\newline By contradiction, if $\Gamma_f$ is not the whole graph,
then there is a loop not in $\Gamma_f$ that has length less
than one by Theorem~\ref{greengraphthm} and that gives a
contradiction with our assumption (\ref{ipotesi}). Hence,
$\Gamma_f = R$.
\newline In order to prove that $f$ is an isometry, we need to show
that we do not have any illegal turn.
\newline If a loop contains an illegal turn,
then its length is stretched by a number $<\mbox{Lip}(f)$ (see \cite{FM}).
Therefore, if we have a loop with an illegal turn,
then the length of that loop would be less than the Lipschitz
constant, $1$, but that is a contradiction.
\newline If $s(j)=2$, then we have other two possibilities for illegal turns (see
Figure~\ref{fig:loop2}). In both cases the length of the figure eight graph would be less than
$2$ and that leads to a contradiction. In general, if we have an illegal
turn in a path contained in a subrose of $m$ petals, then the sum of the
lengths of the edges in the subrose would be less than $m$, and that
contradicts (\ref{ipotesi}).
\newline In conclusion, $f$ is an
isometry and $\mathrm{Minset}(f_j)$ consists of the star of a single rose in
$S_{s(j)}$.

\begin{figure}[htbp]
\begin{center}
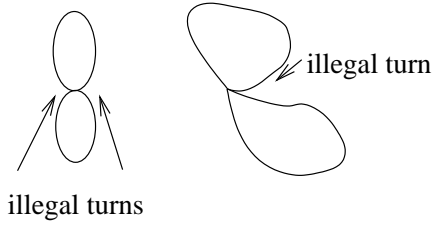

\caption{Two possible illegal turns.} \label{fig:loop2}
\end{center}
\end{figure}

\end{proof}
A different proof of that lemma can be found in \cite{CV}.
\newline Applying Lemma~\ref{starose} to $\mathrm{Minset}(f_j) \hookrightarrow S_n$, for each
$(\Gamma,\psi)$ in $\mathrm{Minset}(f_j)$, the marking $\psi$ is an embedding
on the wedge cycle $\mathbb{B}_j$, and we have the following
corollary.
\begin{corollary} \label{costarose}
$\mathrm{Minset}(f_j) = S_n(A_j)$.
\end{corollary}
Therefore, by Corollary~\ref{costarose} and Lemma~\ref{starose}, if $k=1$ (i.e., we
have only one wedge cycle), then $S_n(A_1)$ is contractible.
\newline Note that $\mathrm{Minset}(f) \subseteq \mathrm{Minset}(f_j)$, for all $j=1, \ldots,
k$.
\newline Now it remains to prove that $S_n(\mathcal{A})$ is contractible for $k>1$.
We will follow the approach described in \cite{BCV}.
\begin{definition}
A forest $F$ in $(\Gamma,\phi) \in \mathrm{Minset}(f)$ is called
\emph{admissible} if the marked graph $(\Gamma',\phi')$ obtained by
collapsing each tree in $F$ to a point is also in $\mathrm{Minset}(f)$.
\end{definition}
\begin{lemma}\label{lemmaimpo}
Let $(\Gamma,\phi) \in \mathrm{Minset}(f)$ and $\phi(C_i^j)$ be the reduced
path representing $\phi(w_i^j)$, for $1 \leq i \leq s(j)$, $1 \leq j
\leq k$. Then
\begin{itemize}
  \item $\mathbb{B}_j = \bigvee_{i=1}^{s(j)} \phi(C_i^j)$ is a wedge cycle in $\Gamma$
  for all $1 \leq j \leq k$;
  \item $\mathbb{B}_j \cap \mathbb{B}_{j'}$ ($j' \neq j$) is either empty, a point or a
  tree;
  \item $\bigcup (\mathbb{B}_j \cap \mathbb{B}_{j'})$ is a forest in $\Gamma$;
  \item If $F$ is an admissible forest in $\Gamma \setminus \{ \phi(C_1^1), \dots, \phi(C_{s(k)}^
  {k}) \}$, then $F \cup \bigcup (\mathbb{B}_j \cap \mathbb{B}_{j'})$ is an admissible
  forest in $\Gamma$.
\end{itemize}
\end{lemma}
\begin{proof}
Let $(R,\psi)$ be any marked rose in $\mathrm{Minset}(f)$ with $(\Gamma,
\phi)$ in its star. By Lemma~\ref{starose}, for each $(\Gamma,
\phi)$ in $\mathrm{Minset}(f)$, the marking $\phi$ is an embedding on
$\bigvee_{i=1}^{s(j)} \phi(C_i^j)$. Thus, $\mathbb{B}_j =
\bigvee_{i=1}^{s(j)} \phi(C_i^j)$ is a wedge cycle in $\Gamma$ for
all $1 \leq j \leq k$.
\newline Because $(\Gamma, \phi)$ is obtained by
blowing up the vertex in $R$ into a tree $T$, the intersection
$\mathbb{B}_j \cap \mathbb{B}_{j'}$ is contained in $T$ and it is
connected (see Figure~\ref{fig:contra1}), so the union of all such
intersections is a forest in $T$. The last statement of the lemma
follows from the previous observation.

\begin{figure}[htbp]
\begin{center}
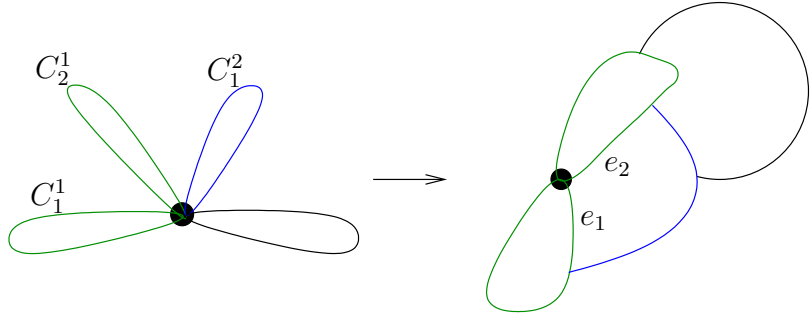

\caption{A point $(\Gamma,\phi)$ in $\mathrm{Minset}(f)$, where $A_1=
<y_1^1,y_2^1>$ and $A_2=<y_1^2>$. The union of the edges $e_1$ and
$e_2$ is the intersection $\mathbb{B}_1 \cap \mathbb{B}_2$.}
\label{fig:contra1}
\end{center}
\end{figure}

\end{proof}
Hence, because $\mathbb{B}_j \cap \mathbb{B}_{j'}$ can be a tree,
$\mathrm{Minset}(f)$ is not contained in $S_n(\mathcal{A})$. However,
we have the following theorem.
\begin{theorem}\label{equal}
$\mathrm{Minset}(f)$ deformation retracts onto $S_n(\mathcal{A})$.
\end{theorem}
\begin{proof}
First of all, notice that we have $S_n(\mathcal{A})
\hookrightarrow \mathrm{Minset}(f)$.
\newline Let $(\Gamma,\phi) \in \mathrm{Minset}(f)$. Collapsing each component of $\bigcup
(\mathbb{B}_j \cap \mathbb{B}_{j'})$ to a point we obtain a map $g$
from $\mathrm{Minset}(f)$ to $S_n(\mathcal{A})$ (see Figure~\ref{fig:contra2}).
\newline If $(\Gamma',\phi') \in \mathrm{Minset}(f)$ is obtained from $(\Gamma,\phi)$
by collapsing a forest $F$, then $F \cup \bigcup (\mathbb{B}_j \cap
\mathbb{B}_{j'})$ is also a forest in $\Gamma$ by Lemma~\ref{lemmaimpo}. Hence, $g$ is a poset map.
\newline By the Poset Lemma, $g$ is a deformation retraction from
$\mathrm{Minset}(f)$ onto $S_n(\mathcal{A})$.

\begin{figure}[htbp]
\begin{center}
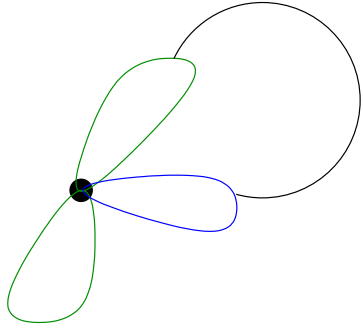

\caption{The image $g((\Gamma,\phi))$, where $(\Gamma,\phi) \in
\mathrm{Minset}(f)$ is the point in Figure~\ref{fig:contra1}, is given by
collapsing $e_1$ and $e_2$.} \label{fig:contra2}
\end{center}
\end{figure}

\end{proof}
We are now able to prove Theorem~\ref{contract}.
\begin{proof}
By Theorem~\ref{equal}, $S_n(\mathcal{A})$ is a deformation
retraction of $\mathrm{Minset}(f)$. Because $\mathrm{Minset}(f)$ is contractible by
Theorem~\ref{contra}, $S_n(\mathcal{A})$ is contractible.
\newline We conclude that $\mathrm{CV}_n(\mathcal{A})$ is contractible
because $\mathrm{CV}_n(\mathcal{A})$ deformation retracts onto $S_n(\mathcal{A})$.
\end{proof}

\section{Relative Spine vs. Small Spine}

We introduce a new spine, called small spine, that is a simplicial
complex smaller than the relative spine, but which carries all the
data coming from the relative outer space.
\newline Consider the relative outer space $\mathrm{CV}_n(\mathcal{A})$.
\begin{definition} \label{def:smallspine}
Let $D_n(\mathcal{A})$ be the subcomplex of $S_n(\mathcal{A})$
spanned by vertices $(\Gamma,\phi)$ in which all the wedge
cycles are disjoint. $D_n(\mathcal{A})$ is called \emph{small
spine}.
\end{definition}
Thus the small spine $D_n(\mathcal{A})$ is a simplicial complex.
The definition of small spine shall be more clear after few
examples.
\begin{example}
Suppose $s(i)>1$ for $i=1, \ldots, k$. Let $(\Gamma, \phi)$ be a
maximal graph in the centroid of a maximal dimensional open
polysimplex in the relative outer space $\mathrm{CV}_n(\mathcal{A})$ with
vertices on the wedge cycles. So the basepoints of the wedge cycles
have valence $2s(1), \ldots, 2s(k)$ and the other vertices have
valence $3$. The maximal simplex of the small spine that contains
$\Gamma$ is given by the barycentric subdivision of the
polysimplex obtained by varying the length of the edges in
$\phi(C_1^1), \ldots, \phi(C_{s(k)}^{k})$ and leaving the length of
the edges in $\Gamma \setminus \{ \phi(C_1^1), \ldots,
\phi(C_{s(k)}^{k}) \}$ equal to $\frac{1}{N}$, where $N$ is the
number of edges in $\Gamma \setminus \{ \phi(C_1^1), \ldots,
\phi(C_{s(k)}^{k}) \}$.
\end{example}
\begin{example}\label{ex2}
Consider the group $\mathrm{Out}(F_4;A_1,A_2)$, where $F_4=<a,a',b,b'>$ and
$A_1=<a,a'>$, $A_2=<b,b'>$. Notice that, using Stalling's method
(see \cite{S}), an element in $\mathrm{Out}(F_4;A_1,A_2)$ is of the form:
$$
\left.
  \begin{array}{ccl}
    a & \mapsto & \omega(a,a') \, a^{\varepsilon_1} \, \overline{\omega}(a,a') \\
    a' & \mapsto & \omega(a,a') \, a'^{\varepsilon_2} \, \overline{\omega}(a,a') \\
    b & \mapsto & \omega(b,b') \, b^{\varepsilon_3} \, \overline{\omega}(b,b') \\
    b' & \mapsto & \omega(b,b') \, b'^{\varepsilon_4} \, \overline{\omega}(b,b') \\
  \end{array}
\right.
$$
where $\varepsilon_i \in \{ \pm 1 \}$, $1 \leq i \leq 4$, $\omega(a,a') \in A_1$ and
$\omega(b,b') \in A_2$.
\newline Modulo separating edges, a point in $\mathrm{CV}_4(A_1,A_2)$ is given by two wedge cycles,
with two cycles each, attached in one point (see Figure~\ref{fig:spineex2}).
\newline The relative spine and the small spine $D_4(A_1,A_2)$
are both equal to the product of two trees
$T_1$ and $T_2$ that correspond to the universal coverings of the
wedge cycles associated to $A_1$ and $A_2$.
\end{example}

\begin{figure}[htbp]
\begin{center}
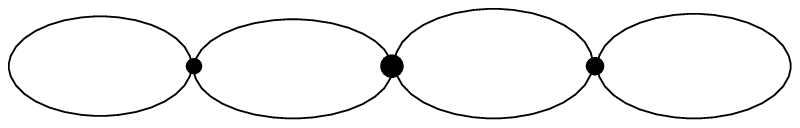
\caption{A point $(\Gamma,\phi)$ in $\mathrm{CV}_4(A_1,A_2)$.}
\label{fig:spineex2}
\end{center}
\end{figure}

\begin{example}\label{ex3}
Consider the group $\mathrm{Out}(F_5;A_1,A_2)$, where $F_5=<a,a',b,b',c>$ and
$A_1=<a,a'>$, $A_2=<b,b'>$ (see Example~\ref{expoly}).
\newline By Example~\ref{ex2} and Stalling's method, an element in $\mathrm{Out}(F_5;A_1,A_2)$ is
of the form:
$$
\left.
  \begin{array}{ccl}
    a & \mapsto & \omega(a,a') \, a^{\varepsilon_1} \, \overline{\omega}(a,a') \\
    a' & \mapsto & \omega(a,a') \, a'^{\varepsilon_2} \, \overline{\omega}(a,a') \\
    b & \mapsto & \omega(b,b') \, b^{\varepsilon_3} \, \overline{\omega}(b,b') \\
    b' & \mapsto & \omega(b,b') \, b'^{\varepsilon_4} \, \overline{\omega}(b,b') \\
    c & \mapsto & u_1(a,a',b,b') c^{\varepsilon_5} u_2(a,a',b,b')
  \end{array}
\right.
$$
where $\varepsilon_i \in \{ \pm 1 \}$, $1 \leq i \leq 5$, $\omega(a,a') \in A_1$,
$\omega(b,b') \in A_2$ and $u_i(a,a',b,b')$'s are elements in $F_4 =
<a,a',b,b'>$.
\newline The relative outer space, the relative spine and the small
spine are more complicated than in Example~\ref{ex2}, but let us
understand what is happening in this case.
\newline Consider the point $(\Gamma,\phi)$ in $\mathrm{CV}_5(A_1,A_2)$
described in Figure~\ref{fig:spineex3}. Varying the length of the
edges in $\Gamma$ we can move in an open (maximal) $5$-polysimplex
of $\mathrm{CV}_5(A_1,A_2)$. When we shrink an edge $e$ that is not in a
wedge cycle, we end up in an open $4$-polysimplex. Because the only
way to move away from that open $4$-polysimplex is to blow up the
vertex given by collapsing $e$, that open $4$-polysimplex is a free
face (i.e., it is a face of a unique polysimplex) in the relative
outer space.
\newline Hence, it is possible to deformation retract the free face
onto the interior of the (maximal) polysimplex.
\newline First collapsing $e$ and then four edges in the wedge
cycles we get a $5$-simplex $\sigma$ in the relative spine. By
Definition~\ref{def:smallspine} of small spine, the edge $e$ cannot
be collapsed and so $\sigma$ is not in the small spine.
\newline Repeating the same argument for the graph in Figure~\ref{fig:expoly1}
of Example~\ref{expoly}, we can compare the simplices in
$S_5(A_1,A_2)$ (see Figure~\ref{fig:barycentric}) with the simplices
in $D_5(A_1,A_2)$ in Figure~\ref{fig:smallspine1}.

\begin{figure}[htbp]
\begin{center}
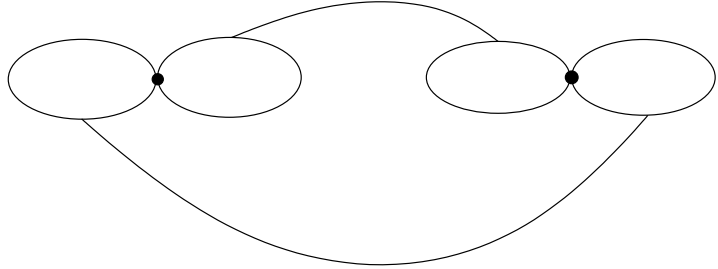

\caption{A point in $S_5(A_1,A_2)$.} \label{fig:spineex3}
\end{center}
\end{figure}

\begin{figure}[htbp]
\begin{center}
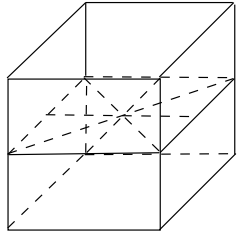

\caption{The simplices in $D_5(A_1,A_2)$ in this picture are given
by the barycentric subdivision of the horizontal square in the middle
of the cube.}
\label{fig:smallspine1}
\end{center}
\end{figure}

Therefore, the small spine is strictly smaller than the relative
spine, but what are missing are vertices in the relative outer space
corresponding to free faces.
\end{example}
Following \cite{BCV} we will prove that there is a deformation
retraction from the relative spine to the small spine.
\begin{theorem}
There is an $\mathrm{Out}(F_n; \mathcal{A})$-equivariant deformation
retraction of $S_n(\mathcal{A})$ onto $D_n(\mathcal{A})$.
\end{theorem}
\begin{proof}
By the definition of $D_n(\mathcal{A})$, we can build $S_n(\mathcal{A})$
from the small spine $D_n(\mathcal{A})$ by adding marked
$(\mathcal{A},n)$-graphs $(\Gamma, \phi)$ in order of decreasing
number of vertices in $\Gamma$. Thus at each stage, we are attaching
$(\Gamma, \phi)$ along its entire upper link in $S_n(\mathcal{A})$.
Hence, it is suffices to show that the upper link is
contractible.
\newline Note that a marked $(\mathcal{A},n)$-graph in $S_n(\mathcal{A})$
with $k$ vertices (the basepoints of the wedge cycles)
of valence $2s(1), \ldots, 2s(k)$ and the remaining vertices of
valence $3$ is in $D_n(\mathcal{A})$.
\newline Let $(\Gamma,\phi) \in S_n(\mathcal{A}) \setminus D_n(\mathcal{A})$.
Then $\Gamma$ contains at least one vertex that is in at
least two wedge cycles. Let $v$ be one of those vertices. In order
to prove that the upper link of $(\Gamma, \phi)$ in $S_n(\mathcal{A})$
is contractible, it is suffices to prove the following
lemma.
\begin{lemma}
If $v$ is contained in at least two wedge cycles, then $L(v)$ is
contractible.
\end{lemma}
That lemma can be proved as the Claim in the
proof of Proposition 17 in \cite{BCV}. We will briefly sketch the
argument of the proof.\\

The set of half edges $E_v$ at $v$ is the union of half edges $A=
\{ a_1, \overline{a}_1, \ldots,$ $ a_r, \overline{a}_r \}$ contained in
some wedge cycle $\mathbb{B}_i$ and $B = \{ b_1, \ldots, b_s \}$ not
contained in any wedge cycle. Fix an element $a \in A$ and define
the \emph{inside} of an ideal edge to be the side containing $a$,
and the \emph{size} to be the number of half edges on the inside.
By hypothesis, $r \geq 2$. The lemma is proved by induction
on $s$. If $s=0$, consider the ideal edge $\alpha$ that
separates $a$ and $\overline{a}$ from all the other half edges. Let
$\mathrm{st}(\alpha)$ denote the star of $\alpha$ in $L(v)$. By adding
vertices of $L(v) \setminus \mathrm{st}(\alpha)$ in order of increasing size,
$L(v)$ deformation retracts onto $\mathrm{st}(\alpha)$. Therefore, $L(v)$ is contractible.
The inductive step can be proved in a similar way.\\

That concludes the proof of the theorem.
\end{proof}
See Figure~\ref{fig:contra3} for an example of the deformation
retraction of $S_n(\mathcal{A})$ onto $D_n(\mathcal{A})$.

\begin{figure}[htbp]
\begin{center}
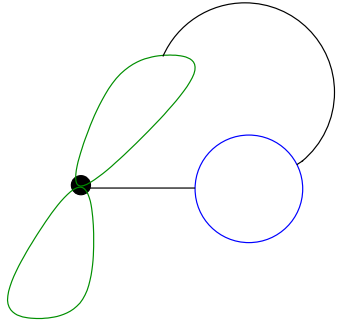

\caption{The deformation retraction of $S_4(A_1, A_2)$ onto
$D_4(A_1,A_2)$ for the graph $(\Gamma,\phi)$ described in Figure~\ref{fig:contra2}.} \label{fig:contra3}
\end{center}
\end{figure}

\begin{corollary}
The small spine $D_n(\mathcal{A})$ is contractible.
\end{corollary}

\section{Virtual Cohomological Dimension of \boldmath $\mathrm{Out}(F_n;\mathcal{A})$}

In this section we obtain a corollary of the fact that the relative
outer space is contractible computing the virtual cohomological
dimension of $\mathrm{Out}(F_n;\mathcal{A})$.
\begin{theorem}
We have
$$
\mathrm{vcd}(\mathrm{Out}(F_n;A_1, \ldots, A_k)) = 2n - 2s(1)- \cdots - 2 s(k) +2k
-2-m,
$$
where $s(i_1) = \cdots = s(i_m) = 1$ and $s(j)>1$ for  $j \neq i_1,
\ldots, i_m$.
\end{theorem}
\begin{proof}
Suppose $s(i_1) = \cdots = s(i_m) = 1$ and $s(j)>1$ for  $j \neq
i_1, \ldots, i_m$.
\newline Recall that we consider $F_n = <y_1^1, \ldots, y_{s(k)}^{k}, x_1,
\dots, x_{n-\sum_{i=1}^{k} s(i)}>$.
\newline We denote $\theta^i = (y_{1}^{i}, \ldots, y_{s(i)}^{i})$.
Reordering the $y$'s if necessary, we can suppose $i_1 =k-m
+1, \ldots, i_m=k$. Consider the quotient map $\mathrm{Aut}(F_n) \rightarrow \mathrm{Out}(F_n)$.
\newline In order to compute the lower bound, we notice that the image
of the Abelian subgroup of $\mathrm{Aut}(F_n)$
$$
A = <\alpha_i, \beta_i, \gamma_j, \delta_r \, | \, 1 \leq i \leq
n-\sum_{i=1}^{k} s(i), 1 < j \leq k, 1 < r \leq k-m
>,
$$
where $\alpha_i$ fixes all the elements of the basis except $x_i
\mapsto y_{1}^{1} x_i$, $\beta_i$ fixes all the elements of the
basis except $x_i \mapsto x_i \bar{y}_{1}^{1}$, $\gamma_j$ fixes all
the elements of the basis except $\theta^j \mapsto y_{1}^{1}
\theta^j \bar{y}_{1}^{1}$ and $\delta_r$ fixes all the elements of
the basis except $\theta^r \mapsto y_{1}^{r} \theta^r
\bar{y}_{1}^{r}$, is in $\mathrm{Out}(F_n;\mathcal{A})$.
\newline Indeed, obviously $\{ \alpha_i \}_{i=1, \ldots, n-\sum_{i=1}^{k}
s(i)}$ and $\{ \beta_i \}_{i=1, \ldots, n-\sum_{i=1}^{k} s(i)}$
commute with all the generators of the subgroup. Because we have
$$
\gamma_i \circ \delta_i(y_j^i)= \gamma_i(y_1^i y_j^i \bar{y_1^i}) =
y_1^1y_1^i \bar{y_1^1} y_1^1 y_j^i \bar{y_1^1} y_1^1 \bar{y_1^i}
\bar{y_1^1} =
$$
$$
= y_1^1 y_1^i  y_j^i  \bar{y_1^i} \bar{y_1^1} = \delta_i(y_1^1 y_j^i
\bar{y_1^1}) = \delta_i \circ \gamma_i (y_j^i),
$$
$\gamma_i$ commutes with $\delta_j$ for all $i$ and $j$. Therefore,
$A$ is Abelian.
\newline It remains to check that all the basis elements are
independent.
Notice that $\{ \alpha_i, \beta_j \}_{i,j=1, \ldots, n-\sum_{i=1}^{k} s(i)}$ are independent (the proof
is analog to the one for $\mathrm{Out}(F_n)$, see \cite{CV}) and that
$\bar{y_1^1} \gamma_j y_1^1$ is the conjugation by $\bar{y_1^1}$ of all the
elements except $\theta^j$. Moreover, any compositions of conjugates
of elements in the free Abelian subgroup $D$ generated by $\alpha_i$
and $\beta_i$ for $1 \leq i \leq n-\sum_{i=1}^{k} s(i)$ cannot be equal to
$\gamma_j$ or $\delta_r$, for all $1 < j \leq k$, $1 < r \leq k-m$.
Indeed, for all $1 < j \leq k$, $1 < r \leq k-m$ we have
\begin{enumerate}
\item $\gamma_j(x_i) = x_i$, $\delta_r(x_i) = x_i$, $\forall \, i \in \{ 1, \ldots, n-\sum_{i=1}^{k} s(i) \}$;
\item $\gamma_j(y_p^1) = y_p^1$, $\delta_r(y_p^1) = y_p^1$, $\forall \, p \in \{ 1, \ldots, s(1) \}$,
\end{enumerate}
and a composition of conjugates of elements in $D$ contradicts $(1)$ or
$(2)$. For example, let $F$ be the conjugate by $y_1^1$ of
$$
\alpha_{1}^{-1} \circ \cdots \circ \alpha_{n-\sum_{i=1}^{k} s(i)}^{-1} \circ \beta_{1}^{-1}
\circ \cdots \circ \beta_{n-\sum_{i=1}^{k} s(i)}^{-1}.
$$
We have $F(x_i) = x_i$, but $F(y_p^1) = y_1^1 y_p^1 \overline{y}_1^1$
for $p \in \{ 1, \ldots, s(1) \}$.

Now we proceed by induction. We start considering the subgroup $D_1$
generated by $\gamma_2$, $\alpha_i$, $\beta_i$ for $1 \leq i \leq n-\sum_{i=1}^{k} s(i)$.
An argument similar to the previous one shows that a
composition of conjugates of elements in $D_1$ cannot be equal to $\gamma_3$.
Hence, $\{ \alpha_i, \beta_l, \gamma_2, \gamma_3 \}$ are independent.
By induction, $\{ \alpha_i, \beta_l, \gamma_j \}$ are independent for
$1 \leq i,l \leq n-\sum_{i=1}^{k} s(i)$, $1 < j \leq k$.

Again using a similar argument, it is easy to prove (by induction) that
$\{ \alpha_i, \beta_l, \gamma_j, \delta_r \}$ are independent for
$1 \leq i,l \leq n-\sum_{i=1}^{k} s(i)$, $1 < j \leq k$, $1 < r \leq k-m$.
\newline Hence, there is an Abelian free group of rank $2n-2 \sum_{i=1}^{k} s(i)+2k-2-m$
contained in our group.
\newline For an upper bound, we compute the dimension of the
small spine. Suppose that the wedge cycles lie in a maximally blown up graph
in the small spine and that the graph has $V$ vertices and $E$
edges. The vertices corresponding to the basepoints of the wedge
cycles have valence $2s(1), \ldots, 2s(k-m)$ and the remaining
vertices have valence $3$.
\newline Therefore,
$$
E  =  \frac{3(V-k+m)}{2} + s(1) +  \cdots + s(k-m) =
$$
$$
= \frac{3V-3k+ 3m+ 2 s(1) +  \cdots + 2 s(k-m)}{2}.
$$
Because $V-E=1-n$, we get
$$
V = 2n + 3k -2s(1)- \cdots -2 s(k-m) - 3m -2.
$$
Because the wedge cycles must stay disjoint, we can collapse $V$ vertices to
$k$ vertices (which is the number of wedge cycles). Then,
$$
\mathrm{dim}(D_n(\mathcal{A})) = 2n + 2k -2s(1)- \cdots -2 s(k-m) - 3m
-2.
$$
Because $s(k-m+1) = \cdots = s(k) =1$,
$$
\left.
  \begin{array}{lll}
    \mathrm{vcd}(\mathrm{Out}(F_n;\mathcal{A})) & \leq & 2n + 2k -2s(1)- \cdots -2 s(k-m)
- 3m -2= \\
     & = & 2n - 2 \sum_{i=1}^{k} s(i) +2k -2-m. \\
  \end{array}
\right.
$$
\newline The result follows from Theorem~\ref{contract}.
\end{proof}
\begin{corollary}
If $m=0$, then
$$
\mathrm{vcd}(\mathrm{Out}(F_n;A_1, \ldots, A_k)) = 2n - 2s(1)- \cdots - 2 s(k) +2k -2.
$$
\end{corollary}
Our computation of the virtual cohomological dimension of the
relative outer space agrees with the computation in \cite{BCV} when
$m=k$. For $k=n$ and $s(1)= \cdots =s(k)=1$, $\mathrm{Out}(F_n;A_1, \ldots, A_k)$ is
called the \emph{pure symmetric automorphism group}. In \cite{C} Collins
showed that the virtual cohomological
dimension of the pure symmetric automorphism group is $n-2$.

\begin{remark}
Suppose that the wedge cycles lie in a maximally blown up graph in
$c$ connected components in the relative spine $S_n(\mathcal{A})$
and that the graph has $V$ vertices and $E$ edges. Because two
wedge cycles must meet in a valence $4$ vertex and the dual graph of
the wedge cycles is a forest, there are $k-c$ vertices of valence
$4$ in the graph. The vertices corresponding to the basepoints of
the wedge cycles have valence $2s(1), \ldots, 2s(k-m)$. The
remaining vertices have valence $3$.
\newline Therefore,
$$
E  =  \frac{3(V-2k+c+m)}{2} + s(1) +  \cdots + s(k-m) + 2(k-c) =
$$
$$
= \frac{3V-2k-c+ 3m+ 2 s(1) +  \cdots + 2 s(k-m)}{2}.
$$
Because $V-E=1-n$, we get
$$
V = 2n + 2k + c -2s(1)- \cdots -2 s(k-m) - 3m -2.
$$
We can collapse $V$ vertices to $1$ vertex by first collapsing all except one edge
in each cycle $\phi(C_i^j)$ and then collapsing some remaining tree.
That gives a simplex of dimension $2n + 2k + c -2 s(1) - \cdots -2
s(k) -3-m$ $(c \leq k)$.
\newline Collapse all the separating edges and note that the maximum
$c$ is $k$ if $n - \sum_{i=1}^{k} s(i) \geq 1$ and $1$ if $n=
\sum_{i=1}^{k} s(i)$.
\newline If $n- \sum_{i=1}^{k} s(i) \geq 1$, then the maximum $c = k$ and the
dimension of a maximal simplex is $2n + 3k -2 s(1) - \cdots -2 s(k)
-3-m$. Notice that if $k=0$, then $m=0$ and
we get the classical result $\mathrm{vcd}(\mathrm{Out}(F_n)) \leq 2n -3$.
\newline If $n= \sum_{i=1}^{k} s(i)$, then the maximum $c
= 1$ and the dimension of a maximal simplex is $2n + 2k -2 s(1) -
\cdots -2 s(k) -2-m$.
\newline Hence,
$$
\mathrm{dim}(S_n(\mathcal{A}))= \left\{
                               \begin{array}{ll}
                                 2n + 3k -2 \sum_{i=1}^{k} s(i)
-3-m, & \mbox{ if } n- \sum_{i=1}^{k} s(i) \geq 1 \\
                                 2n + 2k -2 \sum_{i=1}^{k} s(i)
-2-m, & \mbox{ if } n= \sum_{i=1}^{k} s(i). \\
                               \end{array}
                             \right.
$$
Notice that if $n= \sum_{i=1}^{k} s(i)$, then $\mathrm{dim}(S_n(\mathcal{A}))
 = \mathrm{dim}(D_n(\mathcal{A}))$ (see Example~\ref{ex2}).
\end{remark}
Because a maximal graph has $3n+3k- 2 \sum_{i=1}^{k} s(i)-3-m$ edges
if $n- \sum_{i=1}^{k} s(i) \geq 1$ and $3n+2k -2 \sum_{i=1}^{k} s(i)
-2-m$ edges if $n= \sum_{i=1}^{k} s(i)$ and we impose the conditions
that the relative volume is $1$ (if $n- \sum_{i=1}^{k} s(i) \geq 1$)
and that the sum of the length of the edges in each cycle is $1$, we
have the following result.
\begin{corollary}
$$
\mathrm{dim}(\mathrm{CV}_n(\mathcal{A}))=
\left\{
 \begin{array}{ll}
 3n+3k- 3 \sum_{i=1}^{k} s(i)-4-m, & \mbox{ if } n- \sum_{i=1}^{k} s(i) \geq 1 \\
 3n+2k- 3 \sum_{i=1}^{k} s(i)-2-m, & \mbox{ if } n= \sum_{i=1}^{k} s(i). \\
 \end{array}
\right.
$$
\end{corollary}
Note that if $k=0$ and $n > 1$, then $m=0$ and we have $\mathrm{dim}(\mathrm{CV}_n)=
3n-4$.

\end{document}

%% file: redrose1.pstex_t
\begin{picture}(0,0)%
\includegraphics{redrose1.eps}%
\end{picture}%
\setlength{\unitlength}{3947sp}%
\begingroup\makeatletter\ifx\SetFigFont\undefined%
\gdef\SetFigFont#1#2#3#4#5{%
  \reset@font\fontsize{#1}{#2pt}%
  \fontfamily{#3}\fontseries{#4}\fontshape{#5}%
  \selectfont}%
\fi\endgroup%
\begin{picture}(3129,2440)(2463,-4159)
\put(3473,-3649){\makebox(0,0)[lb]{\smash{{\SetFigFont{12}{14.4}{\rmdefault}{\mddefault}{\updefault}{\color[rgb]{0,0,0}$e_1$}%
}}}}
\put(3973,-4093){\makebox(0,0)[lb]{\smash{{\SetFigFont{12}{14.4}{\rmdefault}{\mddefault}{\updefault}{\color[rgb]{0,0,0}$e_2$}%
}}}}
\put(4736,-3634){\makebox(0,0)[lb]{\smash{{\SetFigFont{12}{14.4}{\rmdefault}{\mddefault}{\updefault}{\color[rgb]{0,0,0}$e_3$}%
}}}}
\put(3451,-3028){\makebox(0,0)[lb]{\smash{{\SetFigFont{12}{14.4}{\rmdefault}{\mddefault}{\updefault}{\color[rgb]{0,0,0}$f_1$}%
}}}}
\put(4842,-3049){\makebox(0,0)[lb]{\smash{{\SetFigFont{12}{14.4}{\rmdefault}{\mddefault}{\updefault}{\color[rgb]{0,0,0}$f_3$}%
}}}}
\put(2478,-2689){\makebox(0,0)[lb]{\smash{{\SetFigFont{12}{14.4}{\rmdefault}{\mddefault}{\updefault}{\color[rgb]{0,0,0}$C_1^1$}%
}}}}
\put(2527,-2146){\makebox(0,0)[lb]{\smash{{\SetFigFont{12}{14.4}{\rmdefault}{\mddefault}{\updefault}{\color[rgb]{0,0,0}$C_2^1$}%
}}}}
\put(5448,-2386){\makebox(0,0)[lb]{\smash{{\SetFigFont{12}{14.4}{\rmdefault}{\mddefault}{\updefault}{\color[rgb]{0,0,0}$C_1^3$}%
}}}}
\put(3931,-1920){\makebox(0,0)[lb]{\smash{{\SetFigFont{12}{14.4}{\rmdefault}{\mddefault}{\updefault}{\color[rgb]{0,0,0}$C_1^2$}%
}}}}
\put(4347,-2894){\makebox(0,0)[lb]{\smash{{\SetFigFont{12}{14.4}{\rmdefault}{\mddefault}{\updefault}{\color[rgb]{0,0,0}$f_2$}%
}}}}
\put(3897,-3350){\makebox(0,0)[lb]{\smash{{\SetFigFont{12}{14.4}{\rmdefault}{\mddefault}{\updefault}{\color[rgb]{0,0,0}$v$}%
}}}}
\put(4657,-1873){\makebox(0,0)[lb]{\smash{{\SetFigFont{12}{14.4}{\rmdefault}{\mddefault}{\updefault}{\color[rgb]{0,0,0}$C_2^2$}%
}}}}
\put(3212,-1919){\makebox(0,0)[lb]{\smash{{\SetFigFont{12}{14.4}{\rmdefault}{\mddefault}{\updefault}{\color[rgb]{0,0,0}$C_3^1$}%
}}}}
\end{picture}%

%% file: exnograph.pstex_t
\begin{picture}(0,0)%
\includegraphics{exnograph.eps}%
\end{picture}%
\setlength{\unitlength}{3947sp}%
\begingroup\makeatletter\ifx\SetFigFont\undefined%
\gdef\SetFigFont#1#2#3#4#5{%
  \reset@font\fontsize{#1}{#2pt}%
  \fontfamily{#3}\fontseries{#4}\fontshape{#5}%
  \selectfont}%
\fi\endgroup%
\begin{picture}(1806,2098)(1730,-3119)
\put(2785,-1175){\makebox(0,0)[lb]{\smash{{\SetFigFont{12}{14.4}{\rmdefault}{\mddefault}{\updefault}{\color[rgb]{0,0,0}$\mathbb{B}_2$}%
}}}}
\put(2782,-3053){\makebox(0,0)[lb]{\smash{{\SetFigFont{12}{14.4}{\rmdefault}{\mddefault}{\updefault}{\color[rgb]{0,0,0}$\mathbb{B}_3$}%
}}}}
\put(1745,-2005){\makebox(0,0)[lb]{\smash{{\SetFigFont{12}{14.4}{\rmdefault}{\mddefault}{\updefault}{\color[rgb]{0,0,0}$\mathbb{B}_1$}%
}}}}
\end{picture}%

%% file: dualgraph1.pstex_t
\begin{picture}(0,0)%
\includegraphics{dualgraph1.eps}%
\end{picture}%
\setlength{\unitlength}{3947sp}%
\begingroup\makeatletter\ifx\SetFigFont\undefined%
\gdef\SetFigFont#1#2#3#4#5{%
  \reset@font\fontsize{#1}{#2pt}%
  \fontfamily{#3}\fontseries{#4}\fontshape{#5}%
  \selectfont}%
\fi\endgroup%
\begin{picture}(1278,1441)(3039,-2528)
\put(3990,-1241){\makebox(0,0)[lb]{\smash{{\SetFigFont{12}{14.4}{\rmdefault}{\mddefault}{\updefault}{\color[rgb]{0,0,0}$\mathbb{B}_2$}%
}}}}
\put(3989,-2462){\makebox(0,0)[lb]{\smash{{\SetFigFont{12}{14.4}{\rmdefault}{\mddefault}{\updefault}{\color[rgb]{0,0,0}$\mathbb{B}_3$}%
}}}}
\put(3054,-1839){\makebox(0,0)[lb]{\smash{{\SetFigFont{12}{14.4}{\rmdefault}{\mddefault}{\updefault}{\color[rgb]{0,0,0}$\mathbb{B}_1$}%
}}}}
\end{picture}%

%% file: expoly1.pstex_t
\begin{picture}(0,0)%
\includegraphics{expoly1.eps}%
\end{picture}%
\setlength{\unitlength}{3947sp}%
\begingroup\makeatletter\ifx\SetFigFont\undefined%
\gdef\SetFigFont#1#2#3#4#5{%
  \reset@font\fontsize{#1}{#2pt}%
  \fontfamily{#3}\fontseries{#4}\fontshape{#5}%
  \selectfont}%
\fi\endgroup%
\begin{picture}(3284,1378)(2454,-3415)
\put(2469,-2222){\makebox(0,0)[lb]{\smash{{\SetFigFont{12}{14.4}{\rmdefault}{\mddefault}{\updefault}{\color[rgb]{0,0,0}$\Gamma$}%
}}}}
\put(5103,-2469){\makebox(0,0)[lb]{\smash{{\SetFigFont{10}{12.0}{\rmdefault}{\mddefault}{\updefault}{\color[rgb]{0,0,0}$s_{1,2}^{2}$}%
}}}}
\put(5723,-2793){\makebox(0,0)[lb]{\smash{{\SetFigFont{10}{12.0}{\rmdefault}{\mddefault}{\updefault}{\color[rgb]{0,0,0}$s_{2,2}^{1}$}%
}}}}
\put(4464,-2909){\makebox(0,0)[lb]{\smash{{\SetFigFont{10}{12.0}{\rmdefault}{\mddefault}{\updefault}{\color[rgb]{0,0,0}$s_{1,2}^{1}$}%
}}}}
\put(3623,-2436){\makebox(0,0)[lb]{\smash{{\SetFigFont{10}{12.0}{\rmdefault}{\mddefault}{\updefault}{\color[rgb]{0,0,0}$s_{2,1}^{2}$}%
}}}}
\put(2532,-2793){\makebox(0,0)[lb]{\smash{{\SetFigFont{10}{12.0}{\rmdefault}{\mddefault}{\updefault}{\color[rgb]{0,0,0}$s_{1,1}^{1}$}%
}}}}
\put(3154,-3037){\makebox(0,0)[lb]{\smash{{\SetFigFont{10}{12.0}{\rmdefault}{\mddefault}{\updefault}{\color[rgb]{0,0,0}$s_{2,1}^{1}$}%
}}}}
\put(3975,-2218){\makebox(0,0)[lb]{\smash{{\SetFigFont{10}{12.0}{\rmdefault}{\mddefault}{\updefault}{\color[rgb]{0,0,0}$t_1$}%
}}}}
\put(4099,-3306){\makebox(0,0)[lb]{\smash{{\SetFigFont{10}{12.0}{\rmdefault}{\mddefault}{\updefault}{\color[rgb]{0,0,0}$t_2$}%
}}}}
\end{picture}%

%% file: excube1.pstex_t
\begin{picture}(0,0)%
\includegraphics{excube1.eps}%
\end{picture}%
\setlength{\unitlength}{3947sp}%
\begingroup\makeatletter\ifx\SetFigFont\undefined%
\gdef\SetFigFont#1#2#3#4#5{%
  \reset@font\fontsize{#1}{#2pt}%
  \fontfamily{#3}\fontseries{#4}\fontshape{#5}%
  \selectfont}%
\fi\endgroup%
\begin{picture}(1461,1454)(2574,-2720)
\put(3247,-2157){\makebox(0,0)[lb]{\smash{{\SetFigFont{10}{12.0}{\rmdefault}{\mddefault}{\updefault}{\color[rgb]{0,0,0}$\Delta_{2}^{1}$}%
}}}}
\put(2881,-1914){\makebox(0,0)[lb]{\smash{{\SetFigFont{10}{12.0}{\rmdefault}{\mddefault}{\updefault}{\color[rgb]{0,0,0}$\sigma$}%
}}}}
\put(2657,-2400){\makebox(0,0)[lb]{\smash{{\SetFigFont{10}{12.0}{\rmdefault}{\mddefault}{\updefault}{\color[rgb]{0,0,0}$\Delta_{1}^{2}$}%
}}}}
\end{picture}%

%% file: spineex1a.pstex_t
\begin{picture}(0,0)%
\includegraphics{spineex1a.eps}%
\end{picture}%
\setlength{\unitlength}{3947sp}%
\begingroup\makeatletter\ifx\SetFigFont\undefined%
\gdef\SetFigFont#1#2#3#4#5{%
  \reset@font\fontsize{#1}{#2pt}%
  \fontfamily{#3}\fontseries{#4}\fontshape{#5}%
  \selectfont}%
\fi\endgroup%
\begin{picture}(6324,2416)(2089,-5838)
\put(4780,-5772){\makebox(0,0)[lb]{\smash{{\SetFigFont{12}{14.4}{\rmdefault}{\mddefault}{\updefault}{\color[rgb]{0,0,0}$\Gamma'$}%
}}}}
\put(6606,-5771){\makebox(0,0)[lb]{\smash{{\SetFigFont{12}{14.4}{\rmdefault}{\mddefault}{\updefault}{\color[rgb]{0,0,0}$\Gamma_1$}%
}}}}
\put(3009,-5759){\makebox(0,0)[lb]{\smash{{\SetFigFont{12}{14.4}{\rmdefault}{\mddefault}{\updefault}{\color[rgb]{0,0,0}$\Gamma$}%
}}}}
\end{picture}%

%% file: barycentersq.pstex_t
\begin{picture}(0,0)%
\includegraphics{barycentersq.eps}%
\end{picture}%
\setlength{\unitlength}{3947sp}%
\begingroup\makeatletter\ifx\SetFigFont\undefined%
\gdef\SetFigFont#1#2#3#4#5{%
  \reset@font\fontsize{#1}{#2pt}%
  \fontfamily{#3}\fontseries{#4}\fontshape{#5}%
  \selectfont}%
\fi\endgroup%
\begin{picture}(1469,536)(1859,-1786)
\end{picture}%

%% file: barycentric1.pstex_t
\begin{picture}(0,0)%
\includegraphics{barycentric1.eps}%
\end{picture}%
\setlength{\unitlength}{3947sp}%
\begingroup\makeatletter\ifx\SetFigFont\undefined%
\gdef\SetFigFont#1#2#3#4#5{%
  \reset@font\fontsize{#1}{#2pt}%
  \fontfamily{#3}\fontseries{#4}\fontshape{#5}%
  \selectfont}%
\fi\endgroup%
\begin{picture}(1538,1505)(2763,-2503)
\end{picture}%

%% file: rphi.pstex_t
\begin{picture}(0,0)%
\includegraphics{rphi.eps}%
\end{picture}%
\setlength{\unitlength}{3947sp}%
\begingroup\makeatletter\ifx\SetFigFont\undefined%
\gdef\SetFigFont#1#2#3#4#5{%
  \reset@font\fontsize{#1}{#2pt}%
  \fontfamily{#3}\fontseries{#4}\fontshape{#5}%
  \selectfont}%
\fi\endgroup%
\begin{picture}(1418,1312)(2219,-2547)
\put(2435,-1389){\makebox(0,0)[lb]{\smash{{\SetFigFont{12}{14.4}{\rmdefault}{\mddefault}{\updefault}{\color[rgb]{0,0,0}$\phi(C_1^j)$}%
}}}}
\put(3622,-1581){\makebox(0,0)[lb]{\smash{{\SetFigFont{12}{14.4}{\rmdefault}{\mddefault}{\updefault}{\color[rgb]{0,0,0}$\phi(C_2^j)$}%
}}}}
\put(3576,-2362){\makebox(0,0)[lb]{\smash{{\SetFigFont{12}{14.4}{\rmdefault}{\mddefault}{\updefault}{\color[rgb]{0,0,0}$\phi(C_3^j)$}%
}}}}
\put(2234,-2397){\makebox(0,0)[lb]{\smash{{\SetFigFont{12}{14.4}{\rmdefault}{\mddefault}{\updefault}{\color[rgb]{0,0,0}$\phi(C_4^j)$}%
}}}}
\end{picture}%

%% file: loop2.pstex_t
\begin{picture}(0,0)%
\includegraphics{loop2.eps}%
\end{picture}%
\setlength{\unitlength}{3947sp}%
\begingroup\makeatletter\ifx\SetFigFont\undefined%
\gdef\SetFigFont#1#2#3#4#5{%
  \reset@font\fontsize{#1}{#2pt}%
  \fontfamily{#3}\fontseries{#4}\fontshape{#5}%
  \selectfont}%
\fi\endgroup%
\begin{picture}(2739,1400)(2394,-2608)
\end{picture}%

%% file: contra1a.pstex_t
\begin{picture}(0,0)%
\includegraphics{contra1a.eps}%
\end{picture}%
\setlength{\unitlength}{3947sp}%
\begingroup\makeatletter\ifx\SetFigFont\undefined%
\gdef\SetFigFont#1#2#3#4#5{%
  \reset@font\fontsize{#1}{#2pt}%
  \fontfamily{#3}\fontseries{#4}\fontshape{#5}%
  \selectfont}%
\fi\endgroup%
\begin{picture}(5042,1964)(1717,-2128)
\put(2971,-658){\makebox(0,0)[lb]{\smash{{\SetFigFont{12}{14.4}{\rmdefault}{\mddefault}{\updefault}{\color[rgb]{0,0,0}$C_1^2$}%
}}}}
\put(1893,-646){\makebox(0,0)[lb]{\smash{{\SetFigFont{12}{14.4}{\rmdefault}{\mddefault}{\updefault}{\color[rgb]{0,0,0}$C_2^1$}%
}}}}
\put(1861,-1446){\makebox(0,0)[lb]{\smash{{\SetFigFont{12}{14.4}{\rmdefault}{\mddefault}{\updefault}{\color[rgb]{0,0,0}$C_1^1$}%
}}}}
\put(5315,-1573){\makebox(0,0)[lb]{\smash{{\SetFigFont{12}{14.4}{\rmdefault}{\mddefault}{\updefault}{\color[rgb]{0,0,0}$e_1$}%
}}}}
\put(5466,-1233){\makebox(0,0)[lb]{\smash{{\SetFigFont{12}{14.4}{\rmdefault}{\mddefault}{\updefault}{\color[rgb]{0,0,0}$e_2$}%
}}}}
\end{picture}%

%% file: contra2a.pstex_t
\begin{picture}(0,0)%
\includegraphics{contra2a.eps}%
\end{picture}%
\setlength{\unitlength}{3947sp}%
\begingroup\makeatletter\ifx\SetFigFont\undefined%
\gdef\SetFigFont#1#2#3#4#5{%
  \reset@font\fontsize{#1}{#2pt}%
  \fontfamily{#3}\fontseries{#4}\fontshape{#5}%
  \selectfont}%
\fi\endgroup%
\begin{picture}(2234,2033)(7313,-2132)
\end{picture}%

%% file: spineex2.pstex_t
\begin{picture}(0,0)%
\includegraphics{spineex2.eps}%
\end{picture}%
\setlength{\unitlength}{3947sp}%
\begingroup\makeatletter\ifx\SetFigFont\undefined%
\gdef\SetFigFont#1#2#3#4#5{%
  \reset@font\fontsize{#1}{#2pt}%
  \fontfamily{#3}\fontseries{#4}\fontshape{#5}%
  \selectfont}%
\fi\endgroup%
\begin{picture}(3832,982)(2091,-4390)
\put(2924,-4311){\makebox(0,0)[lb]{\smash{{\SetFigFont{12}{14.4}{\rmdefault}{\mddefault}{\updefault}{\color[rgb]{0,0,0}$\phi(C_2^1)$}%
}}}}
\put(4139,-4324){\makebox(0,0)[lb]{\smash{{\SetFigFont{12}{14.4}{\rmdefault}{\mddefault}{\updefault}{\color[rgb]{0,0,0}$\phi(C_1^2)$}%
}}}}
\put(5038,-3562){\makebox(0,0)[lb]{\smash{{\SetFigFont{12}{14.4}{\rmdefault}{\mddefault}{\updefault}{\color[rgb]{0,0,0}$\phi(C_2^2)$}%
}}}}
\put(3749,-3613){\makebox(0,0)[lb]{\smash{{\SetFigFont{12}{14.4}{\rmdefault}{\mddefault}{\updefault}{\color[rgb]{0,0,0}$\phi(v)$}%
}}}}
\put(2106,-3564){\makebox(0,0)[lb]{\smash{{\SetFigFont{12}{14.4}{\rmdefault}{\mddefault}{\updefault}{\color[rgb]{0,0,0}$\phi(C_1^1)$}%
}}}}
\end{picture}%

%% file: spineex3.pstex_t
\begin{picture}(0,0)%
\includegraphics{spineex3.eps}%
\end{picture}%
\setlength{\unitlength}{3947sp}%
\begingroup\makeatletter\ifx\SetFigFont\undefined%
\gdef\SetFigFont#1#2#3#4#5{%
  \reset@font\fontsize{#1}{#2pt}%
  \fontfamily{#3}\fontseries{#4}\fontshape{#5}%
  \selectfont}%
\fi\endgroup%
\begin{picture}(4456,1677)(2129,-4839)
\end{picture}%

%% file: smallspine2.pstex_t
\begin{picture}(0,0)%
\includegraphics{smallspine2.eps}%
\end{picture}%
\setlength{\unitlength}{3947sp}%
\begingroup\makeatletter\ifx\SetFigFont\undefined%
\gdef\SetFigFont#1#2#3#4#5{%
  \reset@font\fontsize{#1}{#2pt}%
  \fontfamily{#3}\fontseries{#4}\fontshape{#5}%
  \selectfont}%
\fi\endgroup%
\begin{picture}(1461,1454)(2574,-2720)
\end{picture}%

%% file: contra3a.pstex_t
\begin{picture}(0,0)%
\includegraphics{contra3a.eps}%
\end{picture}%
\setlength{\unitlength}{3947sp}%
\begingroup\makeatletter\ifx\SetFigFont\undefined%
\gdef\SetFigFont#1#2#3#4#5{%
  \reset@font\fontsize{#1}{#2pt}%
  \fontfamily{#3}\fontseries{#4}\fontshape{#5}%
  \selectfont}%
\fi\endgroup%
\begin{picture}(2074,2006)(4485,-2116)
\end{picture}%

%% file: paper1.bbl
\newpage
\begin{thebibliography}{40}

\bibitem{AK} Y. Algom-Kfir, \emph{Strongly contracting geodesics in Outer
Space}, arXiv:0812.1555 math.GR.

\bibitem{BH} M. Bestvina, M. Handel, \emph{Train tracks and automorphisms of free groups},
Annals of Math. (2), \textbf{135} (1992), no. 1, 1--51.

\bibitem{BCV} K. Bux, R. Charney, K. Vogtmann \emph{Automorphisms of two-dimensional RAAGS and partially symmetric automorphisms of free groups},
Groups, Geometry, and Dynamics \textbf{3} (2009) 525--539.

\bibitem{C} D. J. Collins, \emph{Cohomological dimension and symmetric automorphisms of a free group},
Comment. Math. Helv. \textbf{64} (1989), no. 1, 44--61.

\bibitem{CV} M. Culler, K. Vogtmann, \emph{Moduli of graphs and automorphisms of free
groups}, Invent. Math. \textbf{84} (1986), no. 1, 91--119.

\bibitem{FM} S. Francaviglia, A. Martino, \emph{Metric properties of Outer Space},
arXiv:0803.0640v2 math.GR.

\bibitem{HV} A. Hatcher, K. Vogtmann, \emph{Cerf theory for graphs},
J. London Math. Soc. (2) \textbf{58} (1998), no. 3, 633--655.

\bibitem{J} C. A. Jensen, \emph{Contractibility of fixed point sets of outer space},
Topology Appl. \textbf{119} (2002) 287--304.

\bibitem{JW} C. A. Jensen, N. Wahl, \emph{Automorphisms of free groups with
boundaries}, Algebr. Geom. Topol. \textbf{4} (2004) 543--569.


\bibitem{KV} S. Krsti\'{c}, K. Vogtmann, \emph{Equivariant outer space and automorphisms of
free-by-finite groups}, Comment. Math. Helv. \textbf{68} (1993), no.
2, 216--262.

\bibitem{Q} D. Quillen, \emph{Homotopy properties of the poset of $p$-subgroups of a finite
group}, Advances in Math. \textbf{28} (1978) 101--128.

\bibitem{S} J. Stallings, \emph{Topology of finite graphs}, Invent. Math. \textbf{71} (1983)
551--565.

\bibitem{V} K. Vogtmann, \emph{Automorphisms of free groups and Outer Space},
Geometriae Dedicata \textbf{94} (2002) 1--31.



\end{thebibliography}
